\documentclass[graybox]{svmult}

\usepackage{type1cm}
\usepackage{makeidx}
\usepackage{graphicx}
\usepackage{multicol}
\usepackage[bottom]{footmisc}
\usepackage{newtxtext}
\usepackage[varvw]{newtxmath}

\graphicspath{{fig/}}

\usepackage{physics}
\def\div{\operatorname{div}}

\makeindex

\begin{document}

\title*{Mathematical analysis of a partial differential equation system on the thickness}
\titlerunning{Mathematical analysis of a partial differential equation system on the thickness}

\author{Atsushi~Nakayasu\orcidID{0000-0002-2008-7321} and\\ Takayuki~Yamada\orcidID{0000-0002-5349-6690}}
\authorrunning{Atsushi~Nakayasu and Takayuki~Yamada}
\institute{Atsushi~Nakayasu \at Graduate School of Engineering, The University of Tokyo, Yayoi 2-11-16, Bunkyo–ku, Tokyo 113-8656, Japan, \email{ankys@g.ecc.u-tokyo.ac.jp}
\and Takayuki~Yamada \at Graduate School of Engineering, The University of Tokyo, Yayoi 2-11-16, Bunkyo–ku, Tokyo 113-8656, Japan, \email{t.yamada@mech.t.u-tokyo.ac.jp}}

\maketitle

\abstract*{
This study focuses on linear partial differential equation (PDE) systems that arise in topology optimization where the thickness of a structure is constrained.
The thickness derived from the PDE is a fictitious one, and the key challenge of this work is to verify its equivalence to the intuitive, geometrically defined thickness.
The main difficulty lies in that while intuitive thickness is determined solely by the shape, the thickness defined by the PDE depends not only on the shape but also on the entire domain and the diffusion coefficients used in solving the PDE.
In this paper, we demonstrate that the thickness of an infinite, straight film with constant thickness as a simple shape is equivalent within a general domain.
The proof involves constructing a reference solution within a special domain and evaluating the difference using the maximum (modulus) principle and an interior $H^1$ estimate.
Additionally, we provide an estimate of the dependence of thickness on the diffusion coefficient.
}

\abstract{
This study focuses on linear partial differential equation (PDE) systems that arise in topology optimization where the thickness of a structure is constrained.
The thickness derived from the PDE is a fictitious one, and the key challenge of this work is to verify its equivalence to the intuitive, geometrically defined thickness.
The main difficulty lies in that while intuitive thickness is determined solely by the shape, the thickness defined by the PDE depends not only on the shape but also on the entire domain and the diffusion coefficients used in solving the PDE.
In this paper, we demonstrate that the thickness of an infinite, straight film as a simple shape with constant thickness is equivalent within a general domain.
The proof involves constructing a reference solution within a special domain and evaluating the difference using the maximum (modulus) principle and an interior $H^1$ estimate.
Additionally, we provide an estimate of the dependence of thickness on the diffusion coefficient.
}

\section{Introduction}
\label{s:intro}

Structural optimization is a design method that reduces material waste and improves performance by optimizing the size or shape of a structure, or topology of the material under certain conditions.
However, one issue is that the solutions obtained by structural optimization, especially topology optimization, are not always easily feasible in manufacturing.

In recent years, optimization methods that take manufacturability constraints into account have been attracting attention, and thickness adjustment is particularly important.
For example, Allaire, Jouve, and Michailidis propose optimization with thickness constraints based on a signed distance function in \cite{AJM16}.
In addition, Carroll and Guest proposed a method to optimize a structure with a certain thickness using a discrete object projection method \cite{CG22}.
There is also research on constraining thickness using partial differential equations, and Yamada and his coauthors proposed a method to describe geometric features of shapes, including thickness, using a fictitious physical model \cite{Y19, Y19b, SOY22}.

However, the thickness in \cite{Y19, Y19b, SOY22} is a fictitious one defined by a fictitious physical model, and it is introduced independently of the thickness that we intuitively have or that is important in manufacturability.
Therefore, in this study, we aim to investigate whether the thickness defined by this partial differential equation coincides with intuitive thickness.

This topic is by no means trivial and it is difficult because while intuitive thickness is determined by the shape, thickness defined by the partial differential equation is determined not only by the shape domain $\Omega$ but also by the total domain $D$ and the diffusion coefficient $a$ for solving the equation.
In previous work, we have only analytically verified that in one-dimensional space the fictitious thickness coincides with the thickness of the shape (in this case, the length of the section), and numerically verified the thickness obtained by solving the partial differential equation using the finite element method in two-dimensional space \cite{Y19}.
Analytical verification in two or higher dimensions is not well understood, and this study will address this issue.

Another problem is that the definition of intuitive thickness is not clear for general two-dimensional shapes.
Thickness is generally a local concept which varies depending on the position of a point within a shape, and in such cases the problem becomes complicated.
Therefore, in this study, we aim to show it only for specific shapes with a constant thickness, such as a straight film extending infinitely.
The main result (Theorem \ref{t:film}) is that the fictitious thickness calculated by solving partial differential equations from a film shape and a general global domain converges to the thickness of the film when the diffusion coefficient vanishes.

The proof involves the following steps.
First, the general global domain is replaced with a circumscribed membrane-like domain and the reference solution there is analytically obtained.
This calculation is reduced to a one-dimensional problem in space,
and the reference solution is easy to handle.
Next, in view of the maximum principle, an $L^\infty$ estimate of the difference between the solution in the original global domain and the reference solution,
and further, an $H^1$ estimate is obtained by an internal $H^1$ estimate, showing that the fictitious thickness converges to an intuitive thickness.

This paper is organized as follows.
First, in Section \ref{s:eq}, we introduce the linear partial differential equation system, which is the main research subject, and the fictitious thickness defined there.
In Section \ref{s:1dim}, we review the calculation of exact solutions for one-dimensional problems.
In Section \ref{s:film}, we consider the thickness of a film in a global region surrounding it with periodic undulations allowed.
This section states Theorem \ref{t:film} that is the main result of this paper.
Section \ref{s:assum} numerically verifies whether the mathematical assumptions made in Theorem \ref{t:film} are necessary.
Section \ref{s:further} considers how to discuss other shapes such as an annulus.

\section{Partial differential equation system to be studied}
\label{s:eq}

In the $N$-dimensional Euclidean space $\mathbb{R}^N$, prepare a \emph{global domain} $D$ that contains the domain representing the shape $\Omega$ to be analyzed.
The inside of the shape $\Omega$ is called the \emph{shape domain}, and the outside $D\setminus \Omega$ is called the \emph{void domain}.
The boundary between the shape and void domains is denoted by $\Gamma = \partial \Omega$.
To distinguish the shape domain from the void domain, define the characteristic function $\chi \in L^\infty(D)$ by
\[
\chi(x) =
\begin{cases}
1 & \text{if $x \in \Omega$,} \\
0 & \text{if $x \notin \Omega$.} \\
\end{cases}
\]

The linear partial differential equation system to be studied in this paper is of the following form:
\begin{equation}
\label{e:eq}
\begin{cases}
-a\Delta \vb*{s}+(1-\chi)\vb*{s} = -\nabla\chi & \text{in $D$,} \\
\vb*{s} = \vb*{0} & \text{on $\partial D$.} \\
\end{cases}
\end{equation}
Here, $a$ is a positive constant parameter for regularizing the solution $\vb*{s}$, and the case of $a \to 0$ is of most interest.
The solution $\vb*{s}$ is a state variable for extracting the features of the shape $D$, and it is an $N$-dimensional vector-valued function.
For example, it is known that when $a \to 0$, the support of $\vb*{s}$ concentrates on the boundary $\Gamma$ of the shape and converges to the unit normal vector at the point on the boundary \cite{HKTMY20}.
Another example is that a simplified equation is to be used for calculating signed distance functions \cite{HMOSY24}.
In addition, the thickness of the shape, which is the main subject of this paper, is expected to be obtained as the limit of the \emph{(local) thickness function}
\[
h_{\vb*{s}}^a = \frac{2}{\sqrt{a}\div\vb*{s}}
\]
as $a \to 0$.

To solve this equation, we can consider it in weak form as stated below.
A natural regularity of the solution is given by $\vb*{s} \in H^1_0(D)^N$,
that is, each component $s^{x_i}$ of $\vb*{s}$ has square integrable first derivative and satisfies the Dirichlet boundary condition $s^{x_i} = 0$ on $\partial D$.
The weak form is as follows:
\begin{equation}
\label{e:wf}
a\int_{D} \nabla \vb*{s}\colon\nabla \vb*{u}+\int_{D\setminus \Omega}\vb*{s}\cdot \vb*{u} = \int_{\Omega} \div \vb*{u}
\quad \forall \vb*{u} \in H^1_0(D)^N,
\end{equation}
where $\nabla \vb*{s}\colon\nabla \vb*{u}$ is the Frobenius inner product of the Jacobian matricies $\nabla \vb*{s} = (s^{x_i}_{x_j})$ and $\nabla \vb*{u} = (u^{x_i}_{x_j})$,
say $\nabla \vb*{s}\colon\nabla \vb*{u} = \sum_{i, j}s^{x_i}_{x_j}u^{x_i}_{x_j}$.
Note that the right-hand side of the weak form \eqref{e:wf} can be calculated further, and letting $\vb{n}$ be the outward unit normal vector on the boundary of $D$, we have
\[
\int_{D} \div \vb*{u} = \int_\Gamma \vb*{u}\cdot \vb{n}.
\]

This weak form \eqref{e:wf} is said to be a \emph{non-homogeneous} one because the term on the right-hand side exists.
A \emph{homogeneous} equation is one that has no terms on the right-hand side.
\begin{equation}
\label{e:hwf}
a\int_{D} \nabla \vb*{d}\colon\nabla \vb*{u}+\int_{D\setminus \Omega}\vb*{d}\cdot \vb*{u} = 0
\quad \forall \vb*{u} \in H^1_0(D)^N.
\end{equation}
Note that the difference $\vb*{d} = \vb*{s}_2-\vb*{s}_1$ between two solutions $\vb*{s}_1, \vb*{s}_2 \in H^1(D)^N$ of a non-homogeneous equation (which do not necessarily satisfy the boundary conditions) is a solution to a homogeneous equation.

\section{Exact solutions for one-dimensional problems}
\label{s:1dim}

One-dimensional problems when $N = 1$ are easy to handle because the exact solution can be calculated by hand.
Here we consider the cases
\[
D = (b_l, b_r),
\quad \Omega = (f_l, f_r)
\]
with $b_l < f_l < f_r < b_r$.
We set $T = f_r-f_l$, which corresponds to the thickness of the shape $\Omega$.

The solution in this case is exponential inside the void domain $(b_l, f_l)\cup(f_r, b_r)$ because the equation \eqref{e:eq} becomes
\[
-a s''+s = 0 \quad \text{in $(b_l, f_l)\cup(f_r, b_r)$}
\]
and inside the shape domain $(f_l, f_r)$ it becomes a linear one because \eqref{e:eq} becomes
\[
a s'' = 0 \quad \text{in $(f_l, f_r)$}.
\]
Indeed, using integral constants $C_l$ and $C_r$, we have
\[
s(x) =
\begin{cases}
-C_l\sinh \frac{x-b_l}{\sqrt{a}} & \text{for $b_l \le x \le f_l$,} \\
-C_r\sinh \frac{x-b_r}{\sqrt{a}} & \text{for $f_r \le x \le b_r$.} \\
\end{cases}
\]
In the one-dimensional space, since $H^1(D) \subset C^{0, 1/2}(D)$, in $f_l \le x \le f_r$, we simply need to connect the two parts continuously with a linear expression.
The slope at this time is given by
\[
s^* = \frac{C_l\sinh\frac{f_l-b_l}{\sqrt{a}}+C_r\sinh\frac{b_r-f_r}{\sqrt{a}}}{T}
\]

The constants $C_l$ and $C_r$ are determined using the weak form
\[
a\int_{b_l}^{b_r} s'(x)u'(x)\dd{x}+\qty(\int_{b_l}^{b_r}-\int_{f_l}^{f_r}) s(x)u(x)\dd{x}
= \int_{f_l}^{f_r} u'(x)\dd{x}
= u(f_r)-u(f_l)
\]
Here, if we test the piecewise linear function such that
\[
u(b_l) = 0, \quad u(f_l) = 1, \quad u(f_r) = 0, \quad u(b_r) = 0,
\]
we have
\[
a\int_{b_l}^{f_l} s'(x)u'(x)\dd{x}+a\int_{f_l}^{f_r} s'(x)u'(x)\dd{x}+\int_{b_l}^{f_l} s(x)u(x)\dd{x}
= -1.
\]
Note that since $-a s''+s = 0$ holds on $(b_l, f_l)$,
\[
\int_{b_l}^{f_l} s(x)u(x)\dd{x}
= a\int_{b_l}^{f_l} s''(x)u(x)\dd{x}
= [a s'(x)u(x)]_{b_l}^{f_l}-a\int_{b_l}^{f_l} s'(x)u'(x)\dd{x},
\]
\[
[a s'(x)u(x)]_{b_l}^{f_l} = a s'(f_l-0),
\]
\[
a\int_{f_l}^{f_r} s'(x)u'(x)\dd{x} = -a s^*,
\]
where $s'(f_l-0)$ denotes the the left limit of $s'$ at $f_l$.
Therefore,
\[
a s^*-a s'(f_l-0) = 1
\]
is obtained.
Doing the same around $x = f_r$, we have
\[
a s^*-a s'(f_r+0) = 1,
\]
where $s'(f_r+0)$ denotes the the right limit of $s'$ at $f_r$.
From here, $C_l$ and $C_r$ satisfy the simultaneous linear equations
\[
\mqty(\frac{a}{T}\sinh\frac{f_l-b_l}{\sqrt{a}}+\sqrt{a}\cosh\frac{f_l-b_l}{\sqrt{a}} & \frac{a}{T}\sinh\frac{b_r-f_r}{\sqrt{a}} \\ \frac{a}{T}\sinh\frac{f_l-b_l}{\sqrt{a}} & \frac{a}{T}\sinh\frac{b_r-f_r}{\sqrt{a}}+\sqrt{a}\cosh\frac{b_r-f_r}{\sqrt{a}})\mqty(C_l \\ C_r) = \mqty(1 \\ 1).
\]
For simplicity, let $\alpha = \frac{f_l-b_l}{\sqrt{a}}$, $\beta = \frac{b_r-f_r}{\sqrt{a}}$, and $k = \frac{T}{\sqrt{a}}$.
We then have
\[
\mqty(\sinh\alpha+k\cosh\alpha & \sinh\beta \\ \sinh\alpha & \sinh\beta+k\cosh\beta)\mqty(C_l \\ C_r) = \frac{k}{\sqrt{a}}\mqty(1 \\ 1).
\]
This solves
\[
\mqty(C_l \\ C_r)
= \frac{1}{\sqrt{a}}\frac{k}{\sinh(\alpha+\beta)+k\cosh\alpha\cosh\beta}\mqty(\cosh\alpha \\ \cosh\beta).
\]
Therefore, as $a \to 0$ we got
\[
\sqrt{a}s^*
= \frac{1}{T}\frac{k\sinh(\alpha+\beta)}{\sinh(\alpha+\beta)+k\cosh\alpha\cosh\beta}
\to \frac{2}{T}.
\]
As stated at the beginning, the thickness function $h_s^a(x) = \frac{2}{\sqrt{a}s^*}$ converges to the thickness $T = f_r-f_l$.

Moreover, we have the estimate on convergence speed as follows.
\[
\begin{aligned}
h_s^a-T
&= T\qty(\frac{2}{k}+\frac{2\cosh\alpha\cosh\beta}{\sinh(\alpha+\beta)}-1)
= T\qty(\frac{2}{k}+\frac{1-\tanh\alpha+1-\tanh\beta}{\tanh\alpha+\tanh\beta}) \\
&\le T\qty(\frac{2}{k}+\frac{1-\tanh\alpha}{\tanh\alpha}+\frac{1-\tanh\beta}{\tanh\beta})
\le T\qty(\frac{2}{k}+2\exp(-2\alpha)+2\exp(-2\beta)) \\
&= 2\sqrt{a}+2 T\qty(\exp\qty(-2\frac{f_l-b_l}{\sqrt{a}})+\exp\qty(-2\frac{b_r-f_r}{\sqrt{a}})).
\end{aligned}
\]

The contents of this section can be summarized as follows.

\begin{theorem}[Exact solution for one-dimensional space]
When the global domain and shape domain are given by
\[
D = (b_l, b_r),
\quad \Omega = (f_l, f_r),
\]
then for $a > 0$ the function
\[
s(x) =
\begin{cases}
-C_l\sinh \frac{x-b_l}{\sqrt{a}} & \text{for $x \le f_l$,} \\
-C_r\sinh \frac{x-b_r}{\sqrt{a}} & \text{for $x \ge f_r$,} \\
\end{cases}
\]
\[
\mqty(C_l \\ C_r)
= \frac{1}{\sqrt{a}}\frac{k}{\sinh(\alpha+\beta)+k\cosh\alpha\cosh\beta}\mqty(\cosh\alpha \\ \cosh\beta),
\]
with $\alpha = \frac{f_l-b_l}{\sqrt{a}} > 0$, $\beta = \frac{b_r-f_r}{\sqrt{a}} > 0$, $k = \frac{T}{\sqrt{a}}$ is the solution of the equation \eqref{e:eq},
and inside $\Omega$, the thickness function $h_s^a(x) = \frac{2}{\sqrt{a}s'(x)}$ is a constant and its value converges to $T = f_r-f_l$ when $a \to 0$.
Moreover, we have
\[
0 \le h_s^a-T \le 2\sqrt{a}+4 T\exp\qty(-2\frac{m}{\sqrt{a}})
\]
with $m = \min\{ b_r-f_r, f_l-b_l \} > 0$.
\end{theorem}

\section{Straight film shapes}
\label{s:film}

Consider the following domains in $x y$ space with higher dimension, which is periodic in the $x$ direction.
Let $\mathbb{T}^N$ be an $N$-dimensional flat torus $\mathbb{T}^N = \mathbb{R}^N/\mathbb{Z}^N$
and consider
\[
D = \{ (x, y) \in \mathbb{T}^N\times\mathbb{R} \mid b_l(x) < y < b_r(x)\},
\quad \Omega = \mathbb{T}^N\times(f_l, f_r).
\]
In other words, the shape domain $\Omega$ is a straight film, and the boundary of the global domain $D$ is allowed to wavy.
The functions $b_l(x)$ and $b_r(x)$ which define the boundary of the global domain are assumed to be periodic and smooth enough to allow partial integration.

In this case, it turns out that the $x$ component of the solution $\vb*{s} = (s^x, s^y) \in H^1_0(D)^{N+1}$, that is, $s^x$ is identically $0$.
Indeed, putting the test function $\vb*{u} = (u^x, 0)$, $u^x \in H^1_0(D)^N$ into the weak form
\begin{equation}
\label{e:wf2}
a\int_D \nabla \vb*{s}\colon\nabla \vb*{u}+\int_{D\setminus \Omega}\vb*{s}\cdot \vb*{u} = \int_{\Omega} \div \vb*{u}
\quad \forall u \in H^1_0(D)^{N+1},
\end{equation}
we have
\[
a\int_D \nabla s^x\cdot\nabla u^x+\int_{D\setminus \Omega}s^x u^x = \int_{\Omega} \div_x u^x
\quad \forall u^x \in H^1_0(D)^N.
\]
Here, since $\Omega$ is a rectangular domain and is periodic in the $x$ direction,
\[
\int_{\Omega} \div_x u^x = \int_{f_l}^{f_r}\int_{\mathbb{T}^N} \div_x u^x(x, y)\dd{x}\dd{y} = 0.
\]
Therefore, $s^x \in H^1_0(D)^N$ is a solution to a homogeneous equation, so $s^x = 0$.
In the following, we consider the single equation that the $y$ component $s = s^y$ of $\vb*{s}$ satisfies:
\begin{equation}
\label{e:swf}
a\int_D \nabla s\cdot\nabla u+\int_{D\setminus \Omega}s u = \int_{\Omega} u_y
\quad \forall u \in H^1_0(D).
\end{equation}

Here, our goal is to replace the entire wavy domain $D$ with a rectangular domain in order to reduce the higher dimensional problem to one dimension.
For simplicity, let us set
\[
b_l^\land = \min_{\mathbb{T}^N} b_l,
\quad b_l^\lor = \max_{\mathbb{T}^N} b_l,
\quad b_r^\land = \min_{\mathbb{T}^N} b_r,
\quad b_r^\lor = \max_{\mathbb{T}^N} b_r
\]
and consider
\[
\bar{D} = \mathbb{T}^N\times(b_l^\land, b_r^\lor),
\]
which is a domain circumscribing the original global domain $D$.
Solve the system \eqref{e:eq} by applying boundary conditions at its boundaries $\partial\bar{D}$.
Now, the solution $\bar{\vb*{s}}(x, y) = (0, \bar{s}(y))$ can be calculated in the same way as in the one-dimensional case:
\[
\bar{s}(y) =
\begin{cases}
-C_l\sinh \frac{y-b_l^\land}{\sqrt{a}} & \text{for $y \le f_l$,} \\
-C_r\sinh \frac{y-b_r^\lor}{\sqrt{a}} & \text{for $y \ge f_r$,} \\
\end{cases}
\]
\[
\mqty(C_l \\ C_r)
= \frac{1}{\sqrt{a}}\frac{k}{\sinh(\alpha+\beta)+k\cosh\alpha\cosh\beta}\mqty(\cosh\alpha \\ \cosh\beta),
\]
\[
\alpha = \frac{f_l-b_l^\land}{\sqrt{a}} > 0,
\quad \beta = \frac{b_r^\lor-f_r}{\sqrt{a}} > 0,
\quad k = \frac{T}{\sqrt{a}}.
\]
From this,
\[
C_l \le \frac{1}{\sqrt{a}}\frac{1}{\cosh\beta},
\quad C_r \le \frac{1}{\sqrt{a}}\frac{1}{\cosh\alpha},
\]
and hence
we have
\[
-\frac{1}{\sqrt{a}}\frac{\sinh\frac{b_r^\lor-b_r^\land}{\sqrt{a}}}{\cosh\frac{f_l-b_l^\land}{\sqrt{a}}}
\le s(x, y)-\bar{s}(y)
\le \frac{1}{\sqrt{a}}\frac{\sinh\frac{b_l^\lor-b_l^\land}{\sqrt{a}}}{\cosh\frac{b_r^\lor-f_r}{\sqrt{a}}}
\]
for all $(x, y) \in \partial D$.
Therefore, if we assume that the waviness of the boundary $\partial D$ is sufficiently small,
\begin{equation}
\label{e:assum}
R = \min\{ b_r^\lor-f_r, f_l-b_l^\land \}-\max\{ b_r^\lor-b_r^\land, b_l^\lor-b_l^\land \} > 0
\end{equation}
we obtain
\[
\sup_{(x, y) \in \partial D}\abs{s(x, y)-\bar{s}(y)}
\le \frac{1}{\sqrt{a}}\exp\qty(-\frac{R}{\sqrt{a}}) =: C_a
\]
Note that $C_a \to 0$ when $a \to 0$.

In order to obtain the estimate in the domain $D$ from the estimate on the boundary $\partial D$,
we will use the following lemma.

\begin{lemma}[Maximum principle for homogeneous equations]
Let $d \in H^1(D)$ be a weak solution that satisfies the homogeneous equation
\begin{equation}
\label{e:shwf}
a\int_D \nabla d\cdot\nabla u+\int_{D\setminus \Omega}d u = 0
\quad \forall u \in H^1_0(D).
\end{equation}
($d$ does not necessarily satisfy the boundary conditions.)
Then, for a constant $c_- \le 0 \le c_+$, if $c_- \le d \le c_+$ on $\partial D$, then $c_- \le d \le c_+$ on $D$ holds.
\end{lemma}

\begin{remark}
One of the direct consequences of this lemma is that the solution $d$ of the homogeneous equation \eqref{e:shwf} satisfies the maximum modulus principle \cite{KM12}
\[
\sup_{D}\abs{d} \le \sup_{\partial D}\abs{d}.
\]
\end{remark}

\begin{proof}
We only show that $d \le c_+ =: c$.
Let us take the test function
\[
u(x) = k(x)_+ = \max\{ 0, d(x)-c \}.
\]
Since $u \ge 0$ on $D$ and $u \in H^1_0(D)$,
it follows from the weak form \eqref{e:shwf} that
\[
a\int_D \nabla d\cdot\nabla k_+ +\int_{D\setminus \Omega}d k_+ \le 0.
\]
Notine that if $k_+ > 0$, then $d > c \ge 0$ we have
\[
a\int_D \abs{\nabla k_+}^2 \le 0.
\]
Therefore, $\nabla k_+ = 0$ on $D$, so $k_+ = 0$, or $d \le c$.
One can show that $d \ge c_-$ by a similar way.
\end{proof}

Applying the maximum modulus principle to $d(x, y) = s(x, y)-\bar{s}(y)$,
we have the $L^\infty$ estimate
\[
\sup_{(x, y) \in D}\abs{s(x, y)-\bar{s}(y)}
\le C_a.
\]

From here on, we will show the internal $H^1$ estimate.
The internal $H^1$ estimate is mentioned in \cite[Proof of Theorem 1 in Subsection 6.3.1]{E10} and \cite[Problem 8.2]{GT83},
but here we give a complete proof including the coefficients.

\begin{lemma}[Internal $H^1$ estimate for homogeneous equations]
Let $d \in H^1(D)$ be a weak solution that satisfies the homogeneous equation \eqref{e:shwf}.
($d$ does not necessarily satisfy the boundary conditions.)
We then have
\[
\int_{\Omega} \abs{\nabla d}^2
\le \frac{2}{\min\{ b_r^\land-f_r, f_l-b_l^\lor \}^2}\int_{D\setminus \Omega} d^2.
\]
\end{lemma}

\begin{proof}
Note that for $l > 0$ the function
\[
y =
\begin{cases}
0 & (x \le 0), \\
\frac{2}{l^2}x^2 & (0 \le x \le l/2), \\
1-\frac{2}{l^2}(l-x)^2 & (l/2 \le x \le l), \\
1 & (x \ge l), \\
\end{cases}
\]
is a $W^{2, \infty} = C^{1, 1}$ function which satisfies
\[
|y''| \le \frac{4}{l^2}.
\]
This gives us existence of a cutoff function $c \in W^{2, \infty}(\mathbb{R})$ that satisfies the following for $l > 0$.
\[
0 \le c \le 1,
\quad \text{$c = 1$ on $[f_l, f_r]$,}
\quad \text{$c = 0$ on $(-\infty, f_l-l]\cup[f_r+l, \infty)$,}
\]
\[
|c''| \le \frac{4}{l^2}.
\]
Now, since $u(x, y) = c(y)d(x, y)$ is $H^1_0(D)$ for $l < \min\{ b_r^\land-f_r, f_l-b_l^\lor \}$,
it can be used as a test function
and hence
\[
a\int_D \nabla d\cdot\nabla u+\int_{D\setminus \Omega}d u
= a\int_D c\abs{\nabla d}^2+a\int_D d\nabla d\cdot\nabla c+\int_{D\setminus \Omega}c d^2 = 0.
\]
By partially integrating the second term while noting that $d\nabla d = \frac{1}{2}\nabla(d^2)$,
\[
a\int_D c\abs{\nabla d}^2-a\int_D \frac{1}{2}\Delta c d^2+\int_{D\setminus \Omega}c d^2 = 0.
\]
Therefore, we obtain
\[
\int_{\Omega} \abs{\nabla d}^2
\le \int_D \frac{1}{2}\abs{\Delta c}d^2-\frac{1}{a}\int_{D\setminus \Omega} c d^2
\le \frac{2}{l^2}\int_{D\setminus \Omega} d^2
\]
and thus the assersion of this lemma holds.
\end{proof}

In view of this lemma it follows that
\[
\int_{\Omega} \abs{s-\bar{s}}^2
\le \frac{4}{\min\{ b_r^\land-f_r, f_l-b_l^\lor \}}C_a^2.
\]
Therefore, we got the conclusion.

\begin{theorem}[Analysis for the straight film shapes]
\label{t:film}
Assume
\[
D = \{ (x, y) \in \mathbb{T}^N\times\mathbb{R} \mid b_l(x) < y < b_r(x)\},
\quad \Omega = \mathbb{T}^N\times(f_l, f_r)
\]
satisfies \eqref{e:assum}.
Then,
\[
\norm{\frac{1}{h_{\vb*{s}}^a}-\frac{1}{T}}_{L^2(\Omega)}
\le \frac{2}{\sqrt{T}^3}\sqrt{a}+\frac{4}{\sqrt{T}}\exp\qty(-2\frac{\bar{m}}{\sqrt{a}})+\frac{1}{\sqrt{m}}\exp\qty(-\frac{R}{\sqrt{a}})
\]
with $T = f_r-f_l > 0$, $\bar{m} = \min\{ b_r^\lor-f_r, f_l-b_l^\land \} > 0$, $m = \min\{ b_r^\land-f_r, f_l-b_l^\lor \} > 0$ and $R$ as in \eqref{e:assum}.
In particular, we have
\[
\int_{\Omega} \abs{\frac{1}{h_{\vb*{s}}^a}-\frac{1}{T}}^2 \to 0
\]
as $a \to 0$.
\end{theorem}

\section{Remarks on the assumption \eqref{e:assum}}
\label{s:assum}

In the previous section, we show that the thickness function $h_{\vb*{s}}^a$ converges to a constant thickness of the film shape under the assumption \eqref{e:assum},
but it turns out that this assumption is not essential from a numerical point of view.
Indeed, we cannot observe a big gap between the local thickness of numerical solutions by finite element method (FEM) and exact one as in Fig.\ \ref{f:assum}.

In this figure, we calculate the local thickness function by solving the equation \eqref{e:swf} numerically via FEM under the settings where $b_l \equiv 0.0$, $f_l = 0.5$, $f_r = 0.99$, $\max b_r = 3.0$, $\min b_r = 1.0, 2.0, 3.0$ and $a = 0.0001, 0.000001$.
More pricisely, we set
\[
D = \{ (x, y) \in \mathbb{T}\times\mathbb{R} \mid 0.0 < y < 3.0-k\sin(\pi x)^2 \},
\quad \Omega = \mathbb{T}\times(0.5, 0.99)
\]
with $k = 0.0, 1.0, 2.0$.
Note that the cases $k = 1.0, 2.0$ do not satisfy the assumption \eqref{e:assum}.

In the figure, the part where the value of the local thickness is around $0.5$ (greater than $0.3$ and less than $0.7$) is filled in green.
The orange part is the void domain.
Since the constant thickness is $f_r-f_i = 0.49$, we can say that the local thickness approximates the exact thickness even when the assumption \eqref{e:assum} does not hold.

\begin{figure}[t]
\centering
\begin{tabular}{cccc}
 & $\min b_r = 3.0$ & $\min b_r = 2.0$ & $\min b_r = 1.0$ \\
$a = 10^{-4}$
& \includegraphics[width=.23\textwidth,bb=0 0 1707 961]{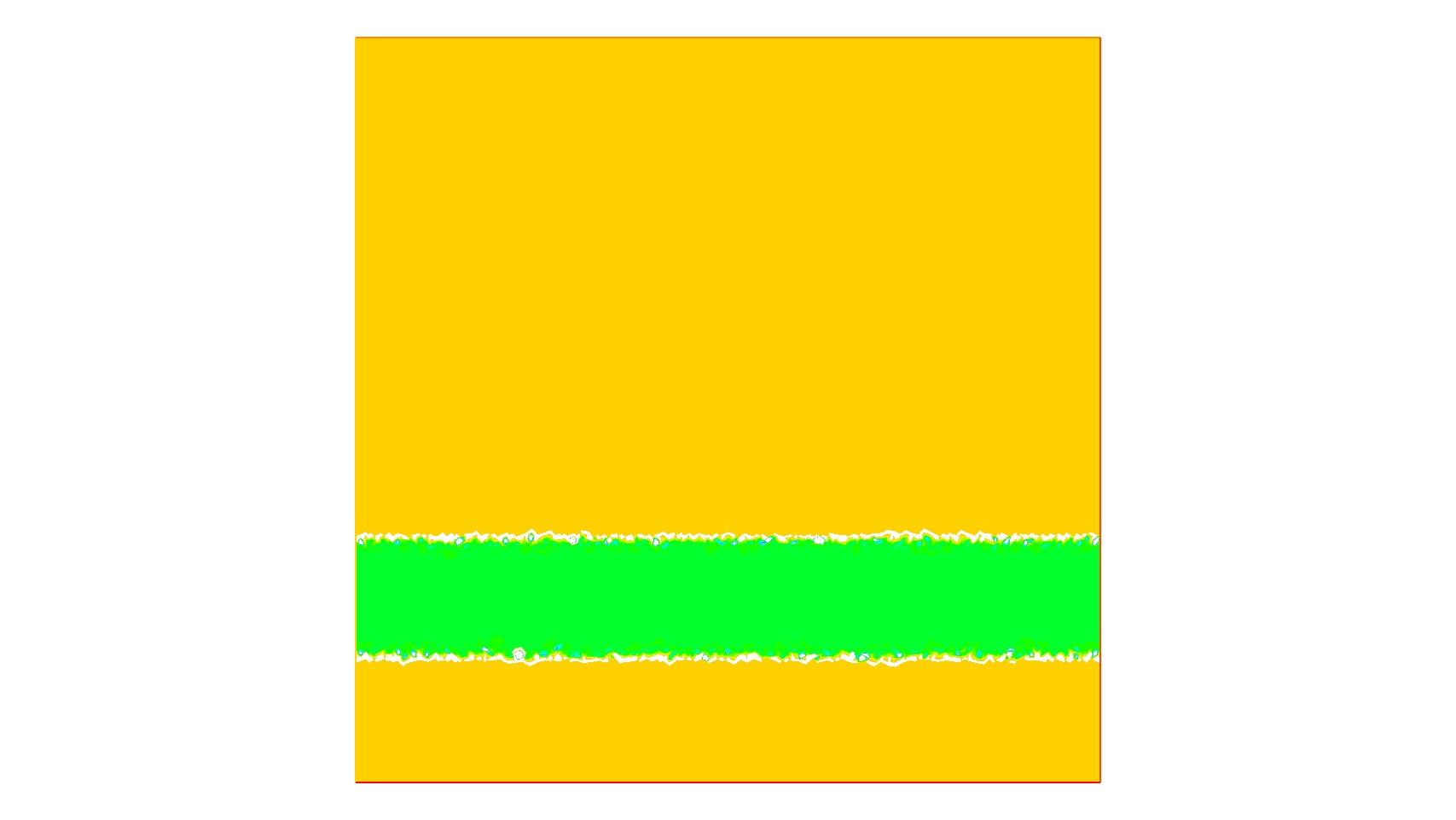}
& \includegraphics[width=.23\textwidth,bb=0 0 1707 961]{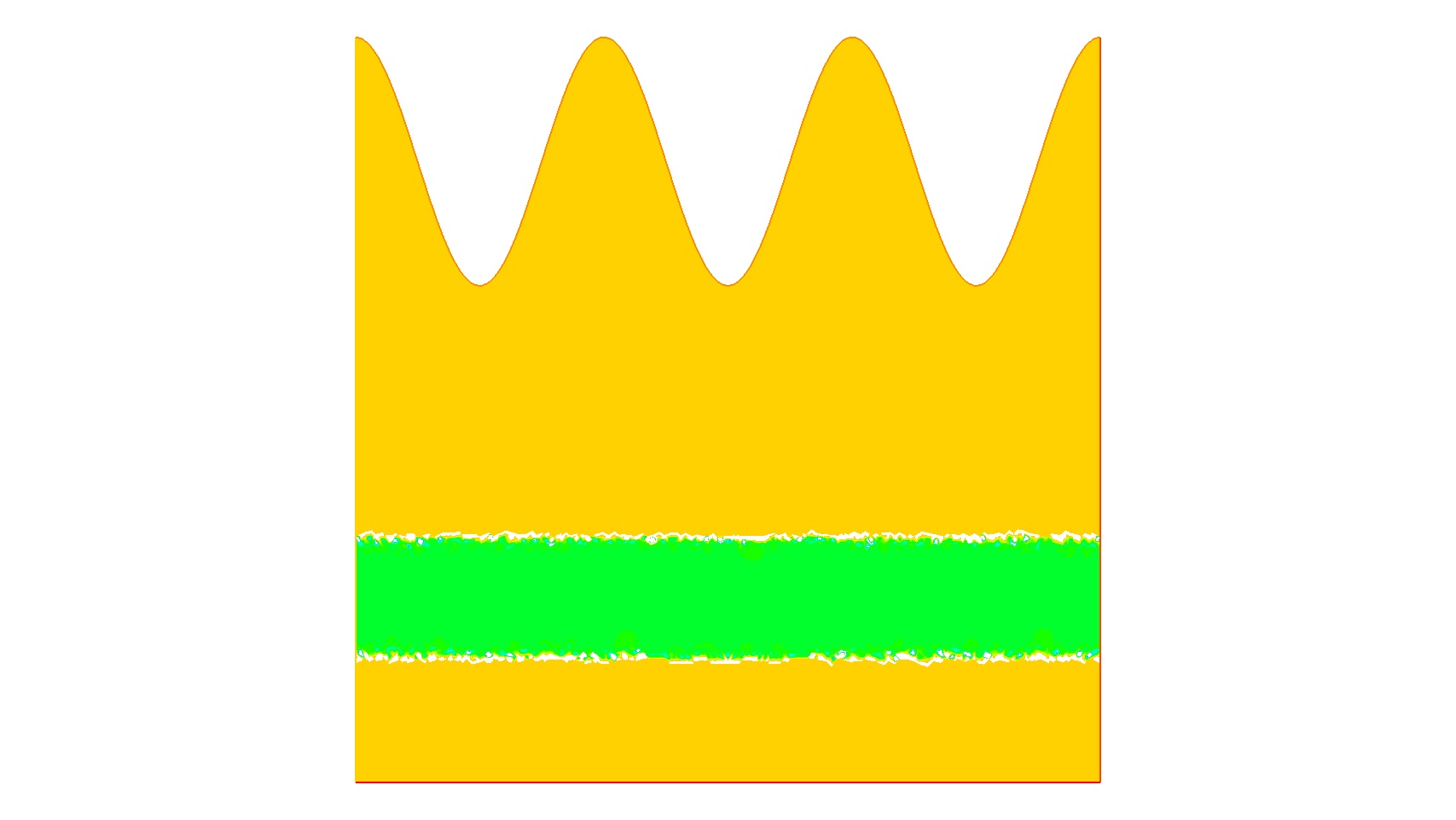}
& \includegraphics[width=.23\textwidth,bb=0 0 1707 961]{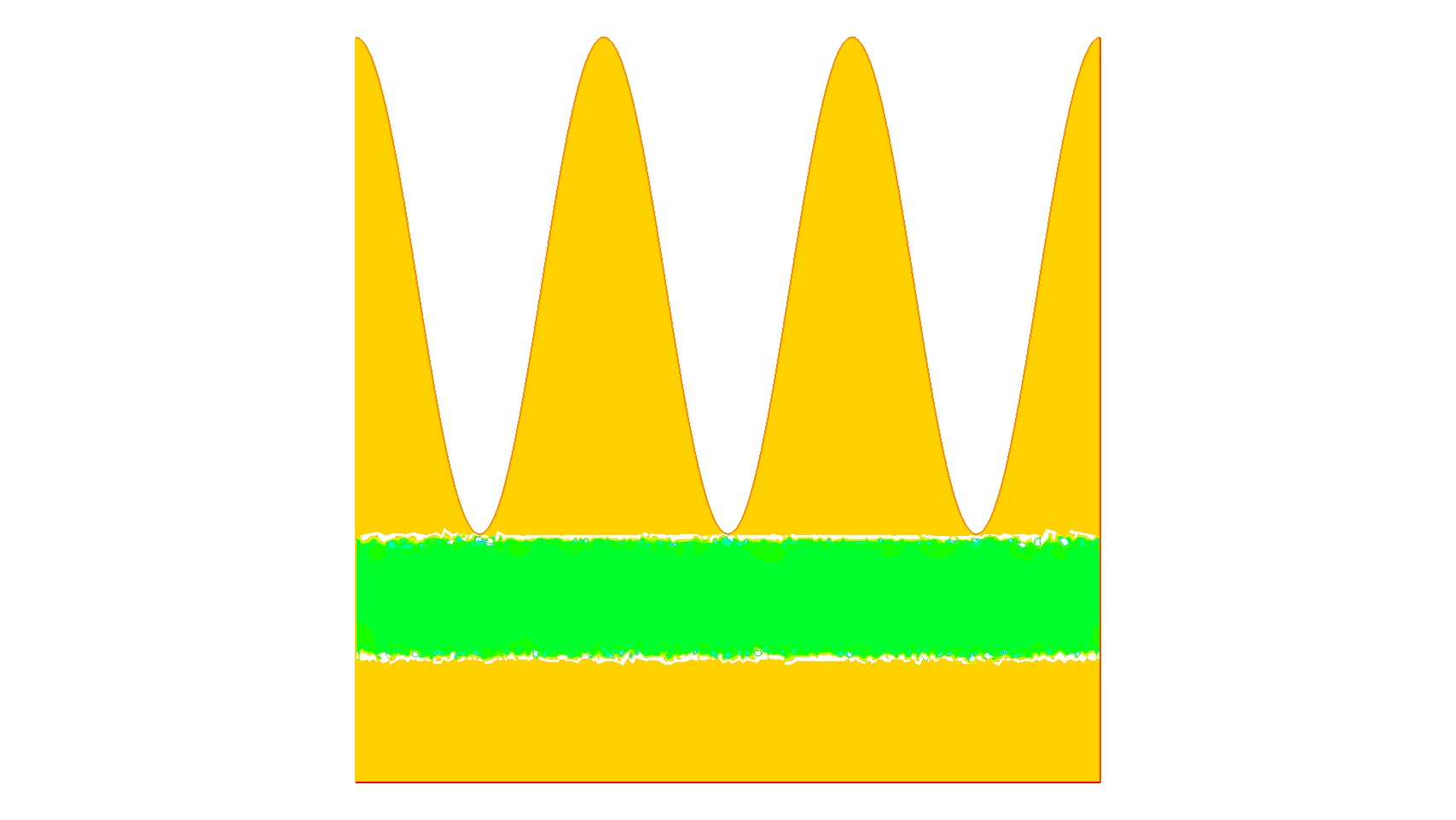} \\
$a = 10^{-6}$
& \includegraphics[width=.23\textwidth,bb=0 0 1707 961]{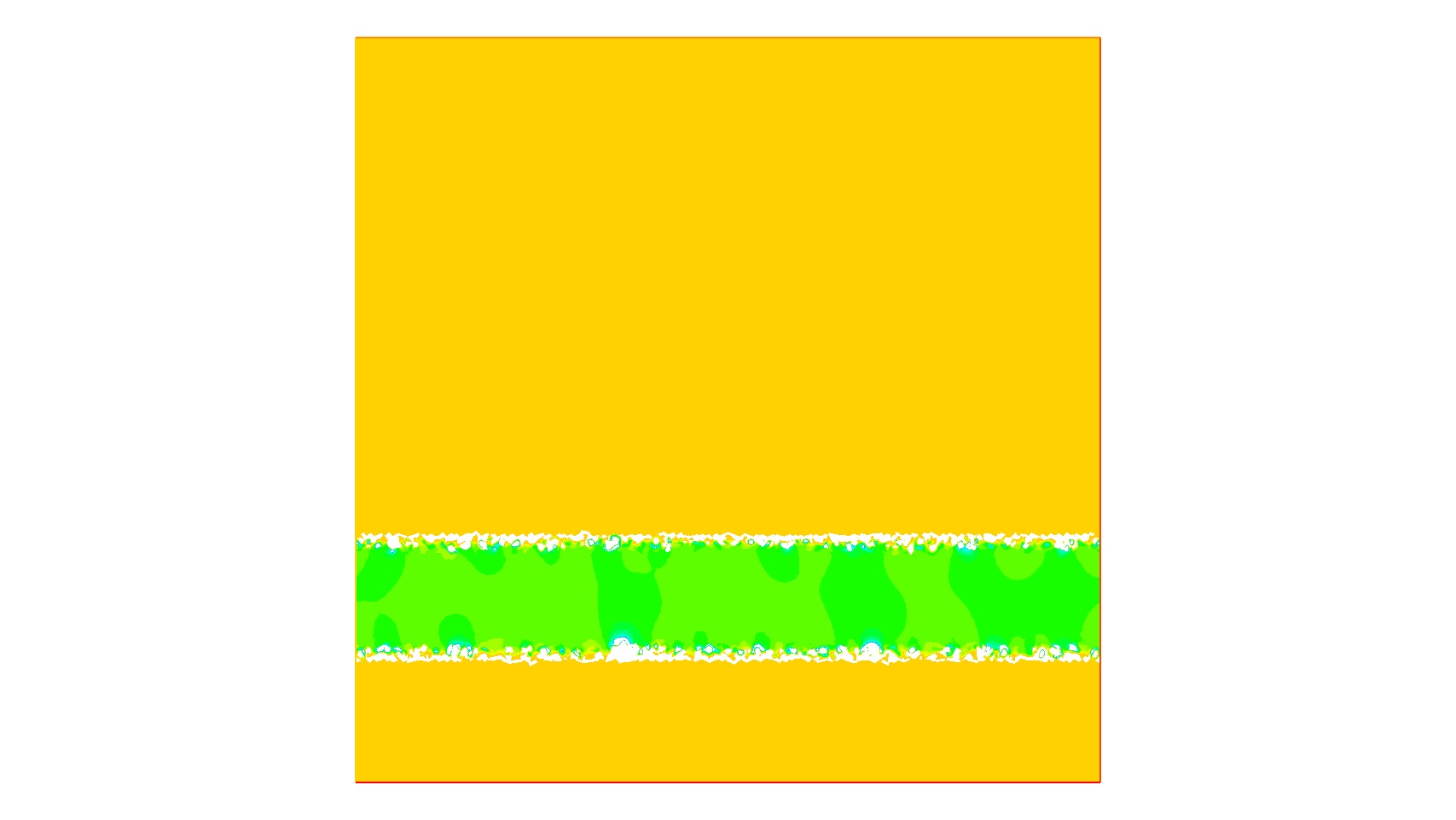}
& \includegraphics[width=.23\textwidth,bb=0 0 1707 961]{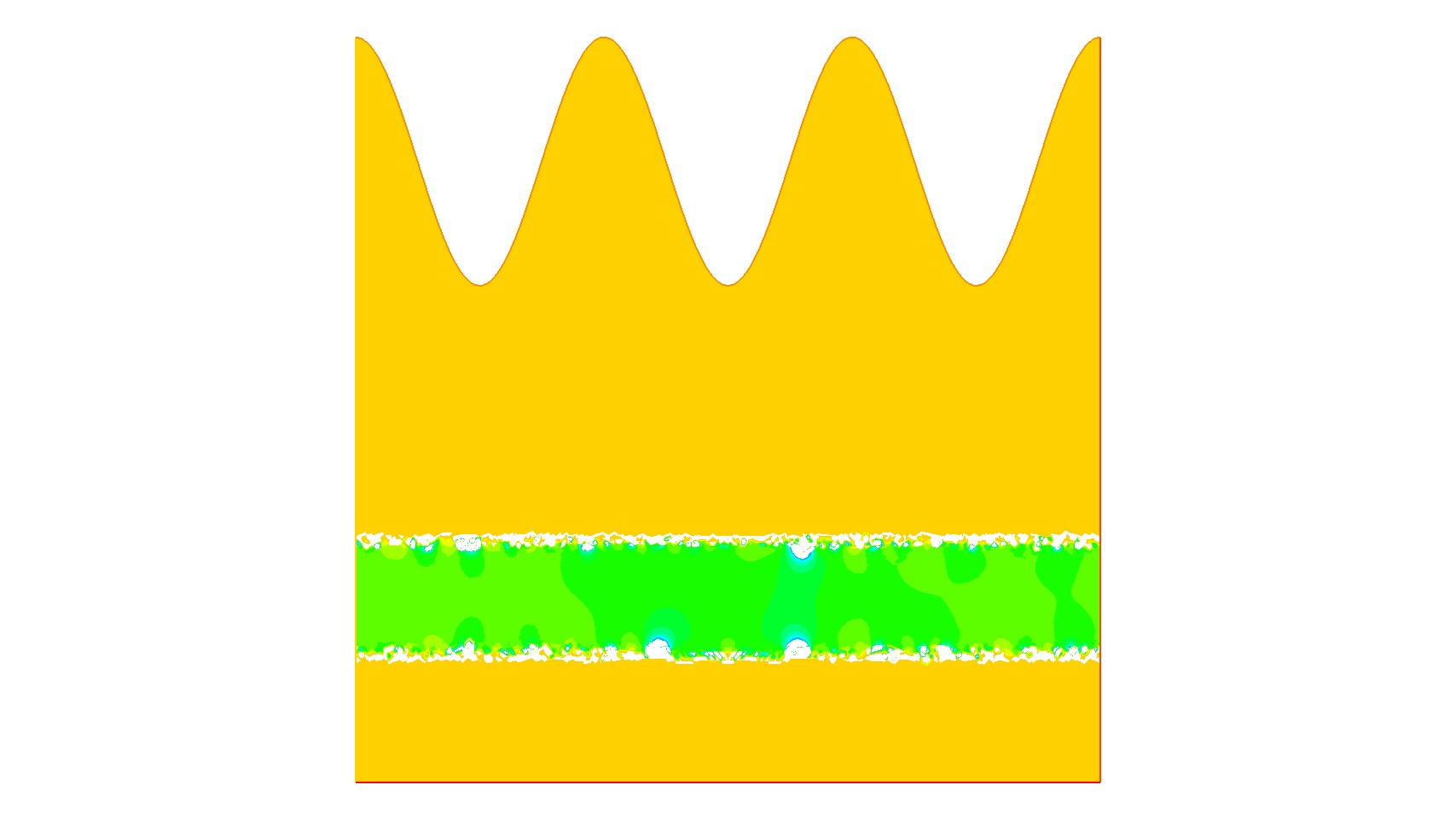}
& \includegraphics[width=.23\textwidth,bb=0 0 1707 961]{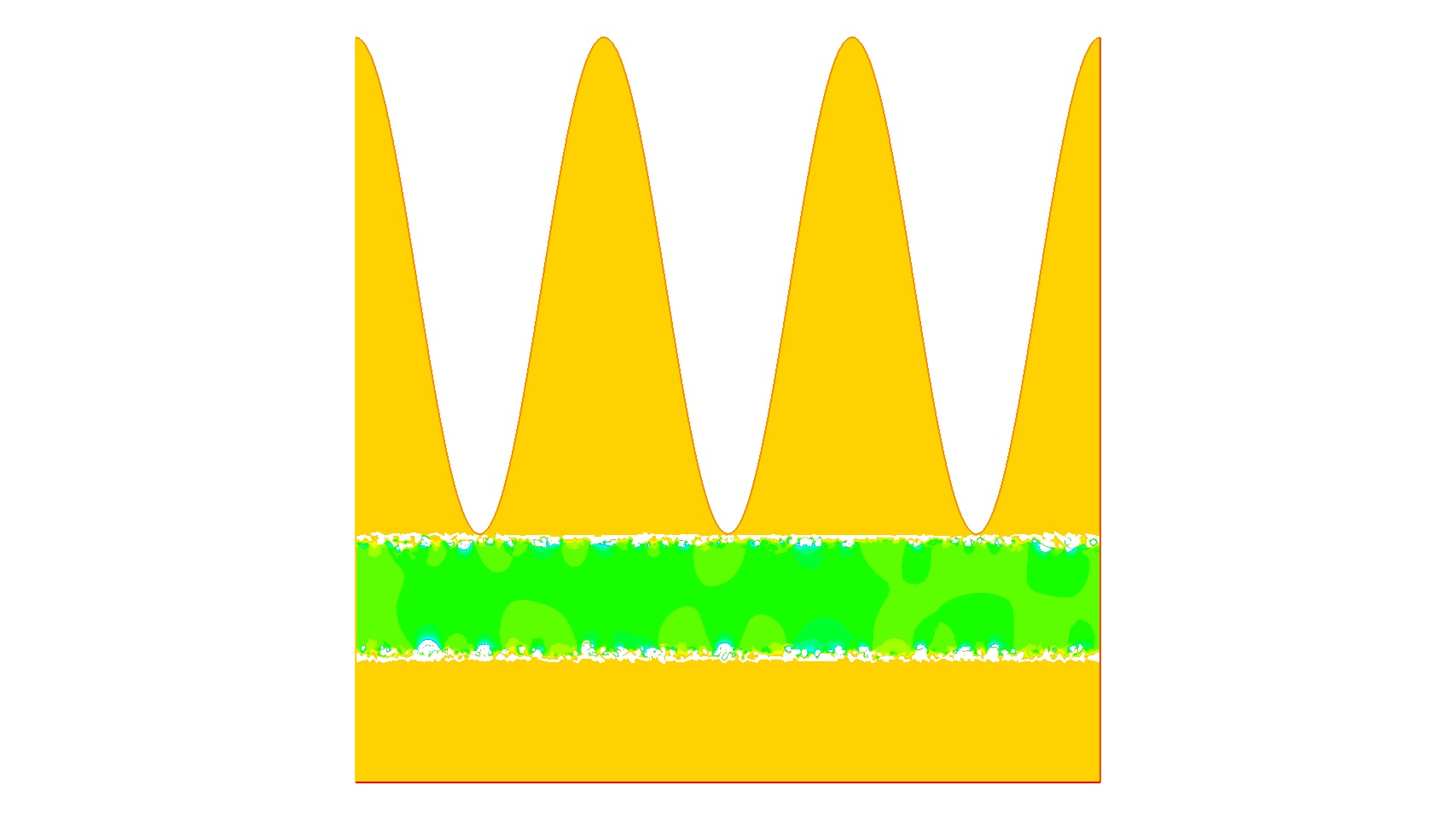} \\
\end{tabular}
\caption{Numerical thickness function via FEM with $b_l \equiv 0.0$, $f_l = 0.5$, $f_r = 0.99$, $\max b_r = 3.0$.}
\label{f:assum}
\end{figure}

\section{Further study}
\label{s:further}

Recall the argument of our main result in Section \ref{s:film}.
The contents of this section are divided into four steps:
\begin{enumerate}
\item
Reduce the system \eqref{e:wf2} to the single equation \eqref{e:swf}.
\item
Construct the reference solution $\bar{s}$.
\item
Get $L^\infty$ estimate via the maximum principle.
\item
Get $H^1$ estimate via the internal $H^1$ principle.
\end{enumerate}

When applying this argument to other shapes such as an annulus,
it may be difficult for us to reduce the system to a single equation.
However, since the maximum principle and the internal $H^1$ estimate can be extended to the system, we can obtain a similar result for other shapes without reduction.

\begin{theorem}[Maximum modulus principle for homogeneous systems]
Let $\vb*{d} \in H^1(D)^N$ be a weak solution of the homogeneous equation \eqref{e:hwf} not necessarily satisfying the boundary conditions.
We then have
\[
\sup_{D}\abs{\vb*{d}} \le \sup_{\partial D}\abs{\vb*{d}}.
\]
\end{theorem}

\begin{proof}
Let $x_0 \in D$ be a one of the maximum points of $\abs{\vb*{d}}$,
and there exists a direction (unit vector) $\vb*{e}$ such that $\vb*{d}(x_0)\cdot \vb*{e} = \abs{\vb*{d}(x_0)}$.
Here, $d^{\vb*{e}}(x) = \vb*{d}(x)\cdot \vb*{e}$ attains its maximum at $x = x_0$ and it is a weak solution of the equation
\[
a\int_{D} \nabla d^{\vb*{e}}\cdot\nabla u+\int_{D\setminus \Omega}d^{\vb*{e}} u = 0
\quad \forall u \in H^1_0(D),
\]
so from the maximum modulus principle for single equations,
we see that $\sup_D \abs{\vb*{d}} \le \sup_{\partial D} \abs{\vb*{d}}$ follows.
\end{proof}

\begin{theorem}[Internal $H^1$ estimate of homogeneous system]
Let $\vb*{d} \in H^1(D)^N$ be a weak solution of the homogeneous equation \eqref{e:hwf} not necessarily satisfying the boundary conditions.
In this case, there exists a constant $C$ that is determined only by $D$ and $\Omega$ such that the inequality
\[
\int_{\Omega} \abs{\nabla \vb*{d}}^2 \le C\int_{D} \abs{\vb*{d}}^2
\]
holds.
\end{theorem}

\begin{proof}
Let us take the nonnegative smooth function $c \in C^\infty_c(D)$ with a compact support such that $c = 1$ on $\Omega$.
In this case, by testing $\vb*{u} = c \vb*{d}$, since $\nabla \vb*{u} = \vb*{d}\otimes\nabla c+c\nabla \vb*{d}$ we have
\[
a\int_D c\abs{\nabla \vb*{d}}^2+a\int_D \nabla \vb*{d}\colon \vb*{d}\otimes\nabla c+\int_{D\setminus \Omega}c\abs{\vb*{d}}^2 = 0.
\]
Here, $\otimes$ represents the tensor product and the identity $\nabla \vb*{d}\colon \vb*{d}\otimes\nabla c = \frac{1}{2}\nabla(\abs{\vb*{d}}^2)\cdot\nabla c$ gives us
\[
a\int_D c\abs{\nabla \vb*{d}}^2-a\int_D \frac{1}{2}\Delta c\abs{\vb*{d}}^2+\int_{D\setminus \Omega}c\abs{\vb*{d}}^2 = 0.
\]
Therefore,
\[
\int_{\Omega} \abs{\nabla \vb*{d}}^2
\le \int_D c\abs{\nabla \vb*{d}}^2
\le \frac{1}{2}\int_D \abs{\Delta c}\abs{\vb*{d}}^2-\frac{1}{a}\int_{D\setminus \Omega}c\abs{\vb*{d}}^2
\le C\int_D \abs{\vb*{d}}^2,
\]
which is nothing but the conclusion of this theorem.
\end{proof}

A detailed calculation is left for future work.


\ethics{Competing Interests}{
This study was funded by JSPS KAKENHI Grant Number JP23H03800.
The authors have no conflicts of interest to declare that are relevant to the content of this chapter.
}

\end{document}